\documentclass{article}
\usepackage{amsmath}
\usepackage{amscd}
\usepackage{amssymb}
\usepackage{amsthm}
\usepackage{enumerate}
\usepackage{pifont}
\usepackage{stmaryrd}
\usepackage{mathrsfs}
\usepackage[dvips]{graphicx}

\input xy
\xyoption{all}
\renewcommand{\thesection}{\arabic{section}}
\newcounter{thmcounter}

\theoremstyle{plain}
\newtheorem{thm}{Theorem}[thmcounter]

\newtheorem{lem}[thm]{Lemma}
\newtheorem{prop}[thm]{Proposition}

\theoremstyle{definition}

\theoremstyle{remark}
\newtheorem{rem}{Remark}[thmcounter]

\usepackage{smallgrf}
\usepackage[all]{xy}

\newcommand{\hooklongrightarrow}{\lhook\joinrel\longrightarrow}
\newcommand{\hooksuperlongrightarrow}{\lhook\joinrel\joinrel\joinrel{-}
\joinrel\joinrel\joinrel\joinrel\joinrel{-}
\joinrel\joinrel\joinrel\joinrel\joinrel{-}
\joinrel\joinrel\joinrel\joinrel\joinrel{-}
\joinrel\joinrel\joinrel\joinrel\joinrel{-}
\joinrel\joinrel\joinrel\joinrel\joinrel{-}
\joinrel\joinrel\joinrel\longrightarrow}

\newcommand{\hookssuperlongrightarrow}{\lhook\joinrel\joinrel\joinrel{-}
\joinrel\joinrel\joinrel\joinrel\joinrel{-}
\joinrel\joinrel\joinrel\joinrel\joinrel{-}
\joinrel\joinrel\joinrel\joinrel\joinrel{-}
\joinrel\joinrel\joinrel\joinrel\joinrel{-}
\joinrel\joinrel\joinrel\joinrel\joinrel{-}
\joinrel\joinrel\joinrel\joinrel\joinrel{-}
\joinrel\joinrel\joinrel\joinrel\joinrel{-}
\joinrel\joinrel\joinrel\joinrel\joinrel{-}
\joinrel\joinrel\joinrel\joinrel\joinrel{-}
\joinrel\joinrel\joinrel\joinrel\joinrel{-}
\joinrel\joinrel\joinrel\joinrel\joinrel{-}
\joinrel\joinrel\joinrel\joinrel\joinrel{-}
\joinrel\joinrel\joinrel\longrightarrow}

\newcommand{\hooksuperlongrrightarrow}{\lhook\joinrel\joinrel\joinrel{-}
\joinrel\joinrel\joinrel\joinrel\joinrel{-}
\joinrel\joinrel\joinrel\joinrel\joinrel{-}
\joinrel\joinrel\joinrel\joinrel\joinrel{-}
\joinrel\joinrel\joinrel\joinrel\joinrel{-}
\joinrel\joinrel\joinrel\joinrel\joinrel{-}
\joinrel\joinrel\joinrel\joinrel\joinrel{-}
\joinrel\joinrel\joinrel\joinrel\joinrel{-}
\joinrel\joinrel\joinrel\joinrel\joinrel{-}
\joinrel\joinrel\joinrel\joinrel\joinrel{-}
\joinrel\joinrel\joinrel\joinrel\joinrel{-}
\joinrel\joinrel\joinrel\joinrel\joinrel{-}
\joinrel\joinrel\joinrel\longrightarrow}

\newcommand{\hooksuperllongrightarrow}{\lhook\joinrel\joinrel\joinrel{-}
\joinrel\joinrel\joinrel\joinrel\joinrel{-}
\joinrel\joinrel\joinrel\joinrel\joinrel{-}
\joinrel\joinrel\joinrel\joinrel\joinrel{-}
\joinrel\joinrel\joinrel\joinrel\joinrel{-}
\joinrel\joinrel\joinrel\joinrel\joinrel{-}
\joinrel\joinrel\joinrel\joinrel\joinrel{-}
\joinrel\joinrel\joinrel\joinrel\joinrel{-}
\joinrel\joinrel\joinrel\joinrel\joinrel{-}
\joinrel\joinrel\joinrel\longrightarrow}

\begin{document}

\title{Minimal fibrations and the organizing theorem of simplicial homotopy theory}

\author{Hiroshi Kihara \\ Center for Mathematical Sciences, University of Aizu\\ Tsuruga, Ikki-machi, Aizu-wakamatsu City, Fukushima, 965-8580, Japan.\\
Tel.: (+81)-242-37-2645\\
Fax: (+81)-242-37-2752\\
E-mail: kihara@u-aizu.ac.jp}
%\thanks{The University of Aizu Center for Mathematical Sciences  \protect \\
%Tsuruga, Ikki-machi, Aizu-Wakamatsu City, Fukushima, 965 Japan. \protect \\
%{\it E-mail address\/}: kihara{\char'100}u-aizu. ac. jp}}
%\institute{Hiroshi Kihara \at
%The University of Aizu Center for Mathematical Sciences \\
%Tsuruga, Ikki-machi, Aizu-wakamatsu City, Fukushima, 965-8580, Japan. \\
%Tel.: (+81)-242-37-2645\\
%Fax: (+81)-242-37-2752\\
%\email{kihara@u-aizu.ac.jp} % \\
%}
%\date{Received: today / Accepted: today}
\date{}
%
%\curraddr{Tsuruga, Ikki-machi, Aizu-Wakamatsu City Fukushima, 965 Japan}

%\topmatter
%\title\nofrills
%Algebraic $K_0$ and $SK_1$-Theories for Rings of Continuous Functions \\
%and Their Homotopical Representations
%\endtitle
%\author
%Hiroshi Kihara
%\endauthor
%\affil
%The University of Aizu\\
%Center for Mathematical Sciences\\
%\endaffil
%\address
%Tsuruga, Ikki-machi, Aizu-Wakamatsu City\\
%Fukushima, 965 Japan
%\endaddress

\maketitle

\begin{abstract}
Quillen showed that simplicial sets form a model category (with appropriate choices of three classes of morphisms), which organized the homotopy theory of simplicial sets. %His method of verification of its model axioms is purely simplicial.
His proof is very difficult and %and not elementary in the sense that it 
uses even the classification theory of principal bundles. Thus%Bousfield-Kan
, Goerss-Jardine %and Hovey
 appealed to topological methods for the verification. % in their books.
 In this paper we give a new proof of this organizing theorem of simplicial homotopy theory
 which is elementary in the sense that it does not use the classifying theory of principal bundles or appeal to topological methods.
\\
\\
%\keywords{Simplicial homotopy theory $\cdot$ Model category $\cdot$ Minimal fibrations}
${\bf Key words:}$ Simplicial sets $\cdot$ Simplicial homotopy theory $\cdot$ Model category $\cdot$ Minimal fibrations
%\keywords{Algebraic $K_0$ $\cdot$ $SK_1$-Theories for Rings of Continuous Functions $\cdot$ Their Homotopical Representations}
\\
$\bf Mathematics$ $\bf Subject$ $\bf Classification\ (2010)$
18G30 - 18G55 - 55U10
\end{abstract}

\section{Introduction}
The homotopy theory of simplicial sets was developed mostly by Kan in the 1950's. 
His work made homotopy theory independent of general topology.
%the category of simplicial sets the first full algebraic model for homotopy theory. 
In the 1960's Quillen [13]
organized the homotopy theory of simplicial sets in the framework of
model categories which is a modern foundation of homotopy theory.
The category of simplicial sets is not only an example of a model category
but also an indispensable ingredient of model category theory (cf. [4-6]).
%; it plays almost the same role in model category theory as the integers do in ring theory
%(cf. [3], [4], and [5]). 
%Since model categories have become more and more important in abstract homotopy theory and its
%applications, simplicial sets and their model category structure are essential in various fields of mathematics 
Thus simplicial sets and their model structure have become essential in various field of mathematics as model categories 
have been important in abstract homotopy theory and its applications [2, 7, 10-12].
% (see, e.g., [8], [6],[9], and [10]).

The organizing theorem of simplicial homotopy theory asserts that simplicial
sets form a model category. By definition a model category is just an ordinary category with three specified
classes of morphisms, called fibrations, cofibrations and weak
equivalences, which satisfy several axioms. Thus the proof of the organizing
theorem is nothing but verification of the model axioms for the
category of simplicial sets and it is, as Goerss and Jardine [4, p. 2]
say,``still one of the more difficult proofs of abstract homotopy theory''
 (see also [6, p. 73]).  
Quillen's method [13] of verification is purely simplicial. But he was ``unable to find a really elementary proof''; 
it uses even the classification theory of principal bundles due to Barratt-Guggenheim-Moore [1] ([13, Introduction of Chapter II]). 
Thus %Bousfield and Kan [2, Chapter VIII], 
Goerss and Jardine [4, Chapter I] %and Hovey [6, Chapter 3] 
appeal to topological methods to show the organizing theorem of simplicial homotopy theory (especially note the definition 
of a weak equivalence [4, p. 60] and the result on the realization of a Kan fibration [4, pp. 54-59]).
%the homotopy theory of simplicial sets is the one which makes homotopy theory independent of general topology, 
But the principle of simplicial homotopy theory is to develope homotopy theory excluding topological methods, hence the proof of the organizing theorem should be purely simplicial but not topological.

Main objective of this paper is to give a %purely simplicial 
proof of the following organizing theorem of simplicial homotopy 
theory which is elementary in the sense that it does not use the classifying theory of principal bundles or appeal to topological methods. %minimizing the prerequisites.

Let {$\bf S$} denote the category of simplicial sets.
If $Z$ is a Kan complex, the homotopy set $[X,Z]$ is defined as the set 
of homotopy classes of maps from $X$ to $Z$ (see Section 2).
\setcounter{section}{1}
%\begin{itshape}
\begin{thm}
%\begin{flushleft}
%$\bf Theorem$ $\bf 1.1.$ 
Define a map $f:X \longrightarrow Y$ in $\bf S$ to be 
%\end{flushleft}
%\newline
%\newline
\newline
\hspace{1em}(1) a fibration if the map $f$ is a Kan fibration,
\newline
\hspace{1em}(2) a weak equivalence if the map $f$ induces a bijection $f^\#: [Y,Z] \longrightarrow [X,Z]$ for any Kan 
\newline
\hspace{2.5em}complex $Z$, and 
%\newline
%\hspace{0.9cm}any Kan complex $Z$, and
\newline
\hspace{1em}(3) a cofibration if the map $f$ has the left lifting property with respect to each map which 
\newline
\hspace{0.9cm}is both a fibration and a weak equivalence.
\newline
Then with these choices $\bf S$ is a model category.
\end{thm}
%\end{itshape}

In fact we show that the category $\bf S$ with the specified classes of morphisms satisfies stronger model axioms, which are adopted in [5] (c.f. Section 2). In other words the category $\bf S$ is complete and cocomplete, and the factorizations described in the so-called factorization axiom are functorially constructed.

%Quillen's proof is not only difficult but also inadequate in the sense that
%it uses even the classification theory of principal
%bundles ([BGM]) to give the
%category of simplicial sets a homotopy-theoretical framework.
%On the other hand, the homotopy theory of simplicial sets is the one which
%makes homotopy theory independent of general topology, hence the proof should be 
%also purely simplicial. From this point of view the proofs appealing to topological methods are also inadequate.
We adopt the definition of a weak equivalence different from those of Quillen [13, 3.14 in Chapter II] and Goerss-Jardine [4, p. 60], which is a key to simplifying the proof of the organizing theorem.
We also check that our model structure coincides with theirs.

%\begin{itshape}
%\begin{flushleft}
\begin{prop}
%$\bf Proposition$ $\bf 1.2.$ 
The model structure of $\bf S$ in Theorem 1.1 coincides with those of Quillen and Goerss-Jardine.
\end{prop}
%\end{flushleft}
%\end{itshape}

The difficulty of checking the model axioms for the category $\bf S$ is having to deal with general simplicial sets which need not to be Kan complexes, which is different from the case of the category of topological spaces. Thus we make effective use of the powerful theory of minimal fibrations whose topological analogue does not exist (see the arguments used in the proofs of Lemmas 2.4 and 3.3). 

The proof of Theorem 1.1 is given in Section 2; a purely simplicial
proof of Lemma 2.4 is crucial. %and the theory of minimal fibrations plays an essential role. 
In section 3, we investigate cofibrations and weak equivalences in the category $\bf S$. The investigation of cofibrations implies Proposition 1.2. An alternative proof of Lemma 2.4, which is the heart of the proof of Theorem 1.1, is given in Section 4. One of the advantage of this proof is that it does not need any knowledge of higher homotopy groups and homotopy exact sequences.

Refer to [4] for the fundamental notions and results of simplicial homotopy
theory and [3] for the basic facts and the terminology of model category
theory. (See also [4] and [5] for more details on model
categories.) The recently published book [9] by May and Ponte is also relevant to our subject. It contains another simplified proof of the organizing theorem of simplicial homotopy theory, which makes a little use of topology.

\section{Proof of Theorem 1.1}

We begin by recalling the standard notations and fundamental notions of the theory of simplicial sets.

Let ${\Delta}^p$ denote the {$\it standard$} {$\it p$}-{$\it simplex$} and ${\Delta}^p_i$ its {$\it i$}-{$\it th$} {$\it face$}. Let $\dot{\Delta}^p$ and ${\Lambda}^p_k$ denote the {$\it boundary$} and the $k$-{$\it th$} {$\it horn$} of ${\Delta}^p$ respectively. Thus $\dot{\Delta}^p$ is just the subcomplex of ${\Delta}^p$ generated by ${\Delta}^p_i\hspace{2mm}(0 \leq i \leq p)$ and ${\Lambda}^p_k$ is just the subcomplex of ${\Delta}^p$ generated by ${\Delta}^p_i\hspace{2mm}(0 \leq i \leq p, i \neq k)$. The standard 1-simplex ${\Delta}^1$ and its boundary $\dot{{\Delta}^1}$ are also denoted by {$I$} and {$\dot{I}$} respectively.

A simplicial map $f:X \longrightarrow Y$ is called a {$\it Kan$} {$\it fibration$} iff $f$ has the right lifting property with respect to standard inclusions ${\Lambda}^p_k \longrightarrow {\Delta}^p \hspace{1mm}(p > 0,\hspace{1mm}0 \leq k \leq p)$. A simplicial set $X$ is called a {$\it Kan$} {$\it complex$} iff the canonical map from $X$ to a terminal object $\ast$ is a Kan fibration. 

Let $f$, $g$ : $X \longrightarrow Y$ be simplicial maps. We say $f$ is ${\it homotopic}$ to $g$, written $f$ $\simeq$ $g$, if there exists a commutative diagram
$$
\begin{picture}(70,70)(-5,-70)
\point{H0}( 0, 0){$X \times \dot{I}$}
\point{H1}( 0, -70){$X \times I$}
\point{H2}( 70,-70){$Y$}
\arrow{H0}{H1} 
\arrow{H0}{H2} \midput( 4, 0){$f+g$}
\arrow{H1}{H2} \midput( 0, 2){$h$}
\end{picture}
$$
where the vertical arrow is the product of $1_X$ and the standard inclusion. The map $h$ is called a ${\it homotopy}$. If $Y$ is a Kan complex, the hotomopy relation $\simeq$ is an equivalence relation ([4, Lemma 6.1 in Chapter I]). Thus 
the ${\it homotopy}$ ${\it set}$ $[X,Y]$ is defined as the quotient set Hom($X$,$Y$)/$\simeq$.

We make full use of the Gabriel-Zisman theory of anodyne extensions ([4, Section 4 in 
Chapter I]) which engulfs old combinatorial arguments.

%\begin{lem}
%The standard inclusions 
%$$\dot{\Delta}^p \times I \cup {\Delta}^p \times (j) \longrightarrow {\Delta}^p \times I\hspace{2mm}(p\geq 0,\hspace{1mm}j = 0,\hspace{1mm}1)$$
%and
%$${\Lambda}^p_k \times \dot{I} \cup {\Delta}^p \times I \longrightarrow {\Delta}^p \times I\hspace{2mm}(p\geq 0,\hspace{1mm}0 \leq k \leq p)$$
%have the left lifting property with respect to all fibrations.
%\end{lem}

%\begin{proof}

%\end{proof}
%\begin{lem}
%Let $\pi:E \longrightarrow {\Lambda}^p_k$ be an $F$-bundle. Then $E$ is trivial.
%\end{lem}

%\begin{proof}
%Since $E$ is trivial over each face ${\Delta}^p_i (i \neq k)$, we can take trivializations $E{\mid}_{{\Delta}^p_i} \xrightarrow[{\cong}]{{\varphi}_i} {\Delta}^p_i \times F $. Note that the complex $Aut(F)$ of automorphisms of $F$ is a Kan complex since it is a group complex ([]). Thus we can modify $\{{\varphi}_i\}$ so that ${\cup}\hspace{0.5mm}{\varphi}_i$ give a trivialization of $E$. 
%\end{proof}

Let us verify the following model axioms which are stronger than Quillen's as stated in Section 1:\\
$\bf MC1$: $\bf S$ is closed under all small limits and colimits.\\
$\bf MC2$: If $f$ and $g$ are maps in $\bf S$ such that $gf$ is defined and if two of the three maps $f$, $g$, 
\newline
\hspace{1.1cm}$gf$ are weak equivalences, then so is the third.\\
$\bf MC3$: If $f$ is a retract of $g$ and $g$ is a weak equivalence, fibration, or cofibration, then so is 
\newline
\hspace{1.1cm}$f$.\\
$\bf MC4$: Given a commutative solid arrow diagram
$$
\xymatrix @r@R=13mm @r@C=17mm{
{A} \ar[d]^i \ar[r] \ar@{}[dr]   & X \ar[d]^{p} \\
 B \ar@^{-->}[ur] \ar[r] & Y \\
}
$$
\\
\hspace{1.1cm}the dotted arrow exists, making the diagram commute, if either\\
\hspace{1.1cm}(i) $i$ is a cofibration and $p$ is an acyclic fibration 
(i.e., a fibration that is also a weak \\ 
\hspace{1.6cm}equivalence), or \\
\hspace{1.1cm}(ii) $i$ is an acyclic cofibration (i.e., a cofibration that is
also a weak equivalence) and \\
\hspace{1.6cm}$p$ is a fibration.\\
$\bf MC5$: Every map $f$ has two functorial factorizations:\\
\hspace{1.1cm}(i) $f$ = $qj$, where $j$ is a cofibration and $q$ is an acyclic 
fibration, and\\ %(i.e., a fibration that \\
%\hspace{1.6cm}is also a weak equivalence), and\\
\hspace{1.1cm}(ii) $f$ = $pi$, where $i$ is an acyclic cofibration and $p$ is a fibration.\\%(i.e., a cofibration that is also a weak\\
%\hspace{1.6cm}equivalence) and $p$ is a fibration.\\

Since ${\bf S} = Set^{{\Delta}^{op}}$, limits and colimits are objectwise constructed. Thus ${\bf MC1}$ is satisfied.\\
${\bf MC2}$ is obvious. ${\bf MC3}$ is not difficult (cf. [3, 2.7 and 8.10]).

To prove the factorization axiom ${\bf MC5}$(ii), note that $f:X \longrightarrow Y$ is a fibration iff $f$ has the right lifting property with respect to standard inclusions ${\Lambda}^p_k \longrightarrow {\Delta}^p \hspace{1mm}(p > 0,\hspace{1mm}0 \leq k \leq p)$.
Since ${\Lambda}^p_k$ is sequentially small, we can apply the small object argument ([3, p. 104]) to ${\mathscr F} = \{{\Lambda}^p_k \longrightarrow {\Delta}^p \hspace{1mm}(p > 0,\hspace{1mm}0 \leq k \leq p)\}$. Then we obtain a factorization
$$
\begin{picture}(70,70)(-5,-50)
\point{H0}( 0, 0){$X$}
\point{H1}( 70, 0){$G^{\infty}(\mathscr F,f)$}
\point{H2}( 70,-70){$Y$}
\arrow{H0}{H1} \midput(  5, 3){\hss{$i_{\infty}$}\hss}
\arrow{H0}{H2} \midput( -5, -5){\cbox{$f$}}
\arrow{H1}{H2} \midput( 0, -5){\hspace{1mm}{$p_{\infty}$}}
\end{picture}
$$
\\
\newline
such that $p_{\infty}$ is a fibration and $G^{\infty}(\mathscr F,f)$ is the direct limit of a sequence of the form 
$$ X = G^{0}(\mathscr F,f)\hspace{1mm} \overset{i_1}{\hooklongrightarrow} 
\hspace{1mm} G^{1}(\mathscr F,f) \hspace{1mm} \overset{i_2}{\hooklongrightarrow} \cdot\cdot\cdot \overset{i_n}{\hooklongrightarrow} \hspace{1mm} G^{n}(\mathscr F,f) \hspace{1mm}{\hooklongrightarrow} \hspace{1mm}
\cdot\cdot\hspace{0.75mm}\cdot 
$$ 
where each $i_n$ fits into a pushout diagram of the form
$$
\begin{picture}(70,70)(0,-60)
\point{H0}( 0, 0){$\underset{\lambda \in \Lambda_n}\coprod{{\Lambda}^{p_\lambda}_{k_\lambda}}$}
\point{H1}( 120, 0){$\underset{\lambda \in \Lambda_n}\coprod{\Delta}^{p_\lambda}$}
\point{H2}( 0, -60){$G^{n-1}(\mathscr F,\hspace{1mm} f)$}
\point{H3}( 120, -60 ){$G^{n}(\mathscr F,\hspace{1mm} f)$}
\point{H4}( 0, 4) { \hspace{2.1cm}}
\point{H5}( 120, 4){ \hspace{2cm}}
%\arrow{H0}{H1}
\arrow{H1}{H3}
\arrow{H0}{H2} 
\point{H6}( 63, -60){$\hooksuperllongrightarrow$} \midput(60, -26){$i_n$} \midput( 150, -35){$.$}
\point{H7}( 61, 3){$\hooksuperlongrrightarrow$} 
%\arrow{H4}{H5}
\end{picture}
$$
Since the class of anodyne extensions is saturated by definition, $i_{\infty}$ is an anodyne extension.
Thus we see that this factorization is the desired one by the following lemma.
\stepcounter{thmcounter}
\begin{lem}
%\begin{flushleft}
%\begin{itshape}
%$\bf Lemma$ $\bf 2.1.$
Let ${i : A \longrightarrow B}$ be an anodyne extension. Then $i$ is an acyclic
cofibration.
%\end{itshape}
%\end{flushleft}
\end{lem}

\begin{proof}
Since ${i : A \longrightarrow B}$ is an anodyne extension, the inclusion
$$A \times I \cup B \times \dot{I} \hooklongrightarrow B \times I$$
is also an anodyne extension by [4, Corollary 4.6 in Chapter I].
Therefore
$$i^{\sharp} : [B, Z] \longrightarrow [A, Z]$$
is bijective for any Kan complex $Z$ (cf. [4, Corollary 4.3 in Chapter I]),
which shows that $i$ is a weak equivalence. Since an anodyne extension is a
cofibration by [4, Corollary 4.3 in Chapter I], $i$ is an acyclic cofibration.
%{\hfill $\Box$ \endtrivlist}
\end{proof}

%It implies that
%$$
%X \times I \hspace{1mm}\cup \hspace{1mm} G^{\infty}(\mathscr F,\hspace{1mm} f) 
%\times \dot{I} \hspace{1mm} \hooklongrightarrow \hspace{1mm} G^{\infty}(\mathscr F,\hspace{1mm} f) 
%\times I
%$$
%is also an anodyne extension by [G-J, Corollary 4.6 in Chapter I]. Therefore
%$$[G^{\infty}(\mathscr F,\hspace{1mm} f),\hspace{1mm} Z] \overset{{i_{\infty}}^{\sharp}}\longrightarrow [X,\hspace{1mm}Z]$$ 
%is bijective for any Kan complex $Z$, which shows that $i_{\infty}$ is a weak equivalence. Since an anodyne extension is a cofibration by [G-J, Corollary 4.3 in Chapter I], $i_{\infty}$ is an acyclic cofibraiton.

Let us make preparations to prove the factorization axiom ${\bf MC5}$(i). A simplicial set $X$ is said to be
${\it fibrant}$ if the map from $X$ to a terminal object $\ast$ is a fibration. Note that $X$ is fibrant
if and only if $X$ is a Kan complex. For each simplicial set 
$X$, we can apply the (functorial) ``acyclic cofibration-fibration'' factorization 
axiom $\bf MC5$(ii) to the map from $X$ to $\ast$ and obtain an acyclic cofibration ${i_X}: X \hookrightarrow \tilde{X}$ with $\tilde{X}$ fibrant. We call $\tilde{X}$ the ${\it fibrant}$ ${\it approximation}$ to $X$. A simplicial map $\pi$ : {$ E \longrightarrow Y$} is called an ${\it F}$-${\it bundle}$ if the pull-back $f^{-1}E$ is isomorphic to ${\Delta^n \times F}$ over $\Delta^n$ for any $n$ and any map $f$ : ${\Delta^n \longrightarrow Y}$.

\begin{lem}
%\begin{itshape}
%\begin{flushleft}
%$\bf Lemma$ $\bf 2.2.$ 
Let $F$ and $Y$ be simplicial sets.\\
%\newline
$\rm (1)$ Let $\pi:E \longrightarrow {\Lambda}^p_k$ be an $F$-bundle. Then $E$ is a trivial $F$-bundle.\\ 
$\rm (2)$ Let $\pi:E \longrightarrow Y$ be an $F$-bundle. Then there exists a
cartesian square
$$
\begin{picture}(60,60)(-5,-50)
\point{H0}( 0, 0) {$E$}
\point{H1}( 60, 0){$E^{\prime}$}
\point{H2}( 0,-60){$Y$}
\point{H3}(60,-60){$\tilde{Y}$}
\point{H4}( 30, 0){$\hooksuperlongrightarrow$}
\point{H5}( 30, 4){$\iota$}
\arrow{H0}{H2} \midput( -8, -5){$\pi$}
\arrow{H1}{H3} \midput( 2, -5){$\pi^{\prime}$}
\point{H6}( 30,-60){$\hooksuperlongrightarrow$}
\point{H7}( 32,-55){$i_Y$}
\end{picture}
$$
\\
such that $\pi^{\prime}$ : {$ E^{\prime} \longrightarrow \tilde{Y}$} is an $F$-bundle and $\iota$ is an acyclic cofibration.
\end{lem}
%\end{flushleft}
%\end{itshape}

\begin{proof}
${\rm (1)}$ Since $E$ is trivial over each face ${\Delta}^p_i$ $(i \neq k)$, we can take trivializations $E{\mid}_{{\Delta}^p_i} \xrightarrow[{\cong}]{{\varphi}_i} {\Delta}^p_i \times F $. Consider the simplicial group $Aut(F)$ whose $p$-simplices are the commutative diagrams of the form
$$
\begin{picture}(60,60)(-5,-50)
\point{H0}( 0, 0) {$\Delta^p \times F$}
\point{H1}( 90, 0){$\Delta^p \times F$}
\point{H2}( 45,-45){$\Delta^p$}
\arrow{H0}{H1} \midput( -5, 4){$\varphi$} \midput( -5, -9){$\cong$}
\arrow{H0}{H2}
\arrow{H1}{H2}
\end{picture}
$$
and note that it
 is a Kan complex (cf. [4, p. 12]). Then we can modify $\{{\varphi}_i\}$ so that ${\cup}\hspace{0.5mm}{\varphi}_i$ gives a trivialization of $E$.\\
(2) By (1), we can easily extend $\pi:E \longrightarrow Y$ to an $F$-bundle over $\tilde{Y}$, which is denoted by $\pi^{\prime}$ : {$ E^{\prime} \longrightarrow \tilde{Y}$}. Then {$ E^{\prime}$} is the direct limit of a sequence of the form
$$
E = E(0) \overset{\iota_1} \hooklongrightarrow E(1) \overset{\iota_2} \hooklongrightarrow \cdot\cdot\cdot \overset{\iota_n} \hooklongrightarrow E(n) \hooklongrightarrow \cdot\cdot\cdot 
$$
where each $\iota_n$ fits into a pushout diagram of the form
$$
\begin{picture}(70,70)(0,-60)
\point{H0}( 0, 0){$\underset{\lambda \in \Lambda_n}\coprod{({\Lambda}^{p_\lambda}_{k_\lambda}} \times F)$}
\point{H1}( 120, 0){$\underset{\lambda \in \Lambda_n}\coprod({\Delta}^{p_\lambda} \times F)$}
\point{H2}( 0, -60){$E(n-1)$}
\point{H3}( 120, -60 ){$E(n)$}
\point{H4}( 0, 4) { \hspace{2.1cm}}
\point{H5}( 120, 4){ \hspace{2cm}}
\point{H6}( 60, 4){$\hooksuperlongrightarrow$}
%\arrow{H0}{H1}
\arrow{H1}{H3}
\arrow{H0}{H2} 
\point{H7}(65,-60){$\hookssuperlongrightarrow$}\midput(60, -25){$\iota_n$} \midput( 140, -35){$.$}
%\point{H8}( 0, 0){\rotatebox[origin=c]{90}{$\longrightarrow$}}
%\arrow{H4}{H5}
\end{picture}
$$
Note that ${\Lambda}^{p_\lambda}_{k_\lambda} \times F \hooklongrightarrow {\Delta}^{p_\lambda} \times F $ 
is an anodyne extension by [4, Corollary 4.6 in Chapter I]. %Then we see that 
%$\iota$ : $E \hooklongrightarrow E^{\prime}$ is an acyclic cofibration
% by the argument used in the proof of the axiom ${\bf MC5}$ (ii).
Since the class of anodyne extensions is saturated by definition, $\iota$
is an anodyne extension, and hence an acyclic cofibration by Lemma 2.1.
%{\hfill $\Box$ \endtrivlist}
\end{proof}

\begin{flushleft}
Under our definition of a weak equivalence, we show
\end{flushleft}
\begin{lem}
%\begin{itshape}
%\begin{flushleft}
%$\bf Lemma$ $\bf 2.3.$ 
Let ${f : X \longrightarrow Y}$ be a weak equivalence between Kan complexes. Then $f$ is a homotopy equivalence.
\end{lem}
%\end{flushleft}
%\end{itshape}
\begin{proof}
Since $f$ is a weak equivalence, ${f^{\sharp} : [Y,Z]  \longrightarrow [X,Z]}$ is bijective for any Kan complex $Z$. Set $Z = X$. Then there exists a unique homotopy class $[g]$ such that $f^{\sharp}([g]) = [1_X]$ and hence $g \circ f \simeq 1_X$. Thus $g^{\sharp}$ is a right inverse to ${f^{\sharp} : [Y,Z] \longrightarrow [X,Z]}$ for any Kan complex $Z$. Since $f^{\sharp}$ is bijective, $g^{\sharp} \circ f^{\sharp} = 1$ also holds on $[Y,Z]$. By setting $Z = Y$ and evaluating this identity at $[1_Y]$, we have $f \circ g \simeq 1_Y$.
%{\hfill $\Box$ \endtrivlist}
\end{proof}

We use Lemmas 2.2 and 2.3 to prove the following lemma which is necessary
to apply the small object argument and obtain the desired factorization.

\begin{lem}
%\begin{itshape}
%\begin{flushleft}
%$\bf Lemma$ $\bf 2.4.$ 
A map $f:X \longrightarrow Y$ is an acyclic fibration if and only if $f$ has the right lifting property with respect to standard inclusions $\dot{\Delta}^n \longrightarrow {\Delta}^n \hspace{2mm}(n \geq 0)$.
%\end{flushleft}
%\end{itshape}
\end{lem}

\begin{proof}
($\Longleftarrow$) Suppose that $f$ has the right lifting property with respect to $\{\dot{\Delta}^n \longrightarrow {\Delta}^n \hspace{2mm}(n \geq 0)\}$ . The standard inclusion ${\Lambda}^n_k$ ${\longrightarrow {\Delta}^n}$ is the composite 
$$ {\Lambda}^n_k \longrightarrow \dot{\Delta}^n \longrightarrow {\Delta}^n$$
and ${\Lambda}^n_k$ ${\longrightarrow \dot{\Delta}^n}$ is a pushout of 
${\dot{\Delta}^{n-1} \longrightarrow {\Delta}^{n-1}}$. Hence it follows that
$f$ is a fibration (cf. [4, Lemma 4.1 in Chapter I]). Thus we take a minimal subfibration ${\varphi}:E \longrightarrow Y$ of $f$. Since $E$ is a vertical deformation retract of $X$, $\varphi$ also has the right lifting property with respect to $\{\dot{\Delta}^n \longrightarrow {\Delta}^n \hspace{2mm}(n \geq 0)\}$. Suppose that the fiber $M$ of $\varphi$ has nontrivial ${\pi}_n$ for some $n$ and some base point $m$. Then we can take a nontrivial element [$h$] $\in$ $\pi_n(M,m)$ and consider a lifting problem
$$
\xymatrix @r@R=13mm @r@C=17mm{
{\dot{\Delta}}^{n+1} \ar[d] \ar[r]^v \ar@{}[dr]   & E \ar[d]^{\varphi} \\
{\Delta}^{n+1}  \ar@^{-->}[ur] \ar[r]^u & Y ,\\
}
$$
%$$
%\begin{picture}(70,70)(0,-60)
%\point{H0}( 0, 0){${\dot{\Delta}}^{n+1}$}
%\point{H1}( 60 , 0){$E$}
%\point{H2}( 0, -60){${\Delta}^{n+1}$}
%\point{H3}( 60 , -60 ){$Y,$}
%\arrow{H0}{H1} \midput( 2 , 2){$v$}
%\arrow{H1}{H3} \midput( 2 , 2){$\varphi$}
%\arrow{H0}{H2} 
%\arrow{H2}{H3} \midput( 2 , 2 ){$u$}
%\arrow{H2}{H1}
%\end{picture}
%$$
where $u$ is the constant map to $\varphi(m)$ and $v$ is the sum of $h$ on the $(n+1)$-th face and the constant map to $m$ on the $(n+1)$-th horn. It is clearly unsolvable, which is a contradiction. Thus $M$ is a minimal complex whose homotopy groups are trivial. By the construction and the uniqueness of a minimal subcomplex [4, Proposition 10.3 and Lemma 10.4 in Chapter I], $M$ is a terminal object, hence ${\varphi}:E \longrightarrow Y$ is an isomorphism. Therefore $f:X \longrightarrow Y$ is a homotopy equivalence and hence a weak equivalence.\\
($\Longrightarrow$) Suppose that $f$ is an acyclic fibration. Then, for a given commutative solid arrow diagram
$$
\xymatrix @r@R=13mm @r@C=17mm{
{\dot{\Delta}}^n \ar[d] \ar[r]^v \ar@{}[dr]   & X \ar[d]^{f} \\
{\Delta}^n  \ar@^{-->}[ur] \ar[r]^u & Y,\\
}
$$
we must construct the dotted arrow making the diagram commute. We take a minimal subfibration ${\varphi}:E \longrightarrow Y$ of $f:X \longrightarrow Y$ and write the inclusion $i$. We fix a vertical retraction ${r : X \longrightarrow E}$ and a vertical homotopy ${R : X \times I \longrightarrow X}$ from $1_X$ to $ir$ such that the composite ${E \times I \overset{i \times 1}\hooklongrightarrow X \times I \overset{R}\longrightarrow X}$ is the constant homotopy of $i$.

First let us show that $\varphi$ is an isomorphism. Note that $\varphi$ is a fiber bundle whose fiber $M$ is a minimal complex ([4, p. 47 and Collary 10.8 in Chapter I]). Thus we have an extension of an $M$-bundle $\varphi$
$$
\begin{picture}(60,60)(-5,-50)
\point{H0}( 0, 0) {$E$}
\point{H1}( 60, 0){$E^{\prime}$}
\point{H2}( 0,-60){$Y$}
\point{H3}(60,-60){$\tilde{Y}$}
\point{H4}(30, 0){$\hooksuperlongrightarrow$}
\point{H5}(30,-60){$\hooksuperlongrightarrow$}
\point{H6}(30, 4){$\iota$}
\point{H7}(32,-55){$i_Y$}
\arrow{H0}{H2} \midput( -8, -5){$\varphi$}
\arrow{H1}{H3} \midput( 2, -5){$\varphi^{\prime}$}
\end{picture}
$$
\\
by Lemma 2.2. Since $\varphi^{\prime}$ is a fiber bundle with Kan complex fiber, $\varphi^{\prime}$ is a fibration. Thus both of $\tilde{Y}$ and $E^{\prime}$ are Kan complexes and $\varphi^{\prime}$ is a homotopy equivalence by Lamma 2.3. Hence $M$ has trivial homotopy groups by the homotopy exact sequence ([4, Lemma 7.3 in Chapter I]). This implies that $M$ is a terminal object, which shows that $\varphi$ is an isomorphism.

%By MC(ii) the canonical map $Y \longrightarrow e$ to a final object $e$ is factorized as 
%$$Y \overset{i_{\infty}} \longrightarrow Y^{\prime} \overset{p_{\infty}} \longrightarrow e$$
%where $i_{\infty}$ is a weak equivalence and cofibration and $p_{\infty}$ is a fibration. Thus we construct the following extension
%$$
%\begin{picture}(60,60)(-5,-50)
%\point{H0}( 0, 0) {$E$}
%\point{H1}( 60, 0){$E^{\prime}$}
%\point{H2}( 0,-60){$Y$}
%\point{H3}(60,-60){$Y^{\prime}$}
%\point{H4}( 12.3, -0.6){$\leftharpoonup$}
%\point{H5}( 13.3, -60.6){$\leftharpoonup$}
%\arrow{H0}{H1}
%\arrow{H0}{H2} \midput( -8, -5){$\varphi$}
%\arrow{H1}{H3} \midput( 2, -5){$\varphi^{\prime}$}
%\arrow{H2}{H3}
%\end{picture}
%$$
%\newline
%as an $M$-bundle using Lemma 2.2. 
%By the exponential law and Lemma 2.1, $E \hookrightarrow E^{\prime}$ is a weak equivalence. Since $Y'$ is a Kan complex, all the homotopy groups of $M$ vanish by the homotopy exact sequence ([]), which implies that $M$ is a final object ([]). Hence $\varphi$ is an isomorphism.

Now consider the following lifting problem
$$
\xymatrix @r@R=13mm @r@C=17mm{
\dot{\Delta}^n \times I \cup {\Delta}^n \times (1) \ar[d] \ar[r]^{\hspace{1.3cm}{V}} \ar@{}[dr]   & X \ar[d]^f \\
{\Delta}^n \times I \ar@^{-->}[ur] \ar[r]^{\hspace{3mm}{U}} & Y \\
}
$$
where $U$ is the composite
$$
\begin{picture}(70,0)(25,5)
\point{H0}( 0, 0){${\Delta}^n \times I$}
\point{H1}( 70, 0){${\Delta}^n$}
\point{H2}( 140, 0){$Y$}
\arrow{H0}{H1} \midput( -3 ,5){$proj$}
\arrow{H1}{H2} \midput( -3 ,5){$u$}
\end{picture}
$$   
\begin{flushleft}
and $V$ is the sum of
\end{flushleft}

$$
\begin{picture}(70,0)(-5,0)
\point{H0}( -40, 0){$\dot{\Delta}^n \times I$}
\point{H1}( 30, 0){$X \times I$}
\point{H2}( 100, 0){$X$}
\arrow{H0}{H1} \midput( -10 ,5){$v \times 1$}
\arrow{H1}{H2} \midput( -3 ,4){$R$}
\end{picture}
$$
\begin{flushleft}
and 
\end{flushleft}
$$
\begin{picture}(70,0)(-5,0)
\point{H0}( -40, 0){${\Delta}^n \times (1)$}
\point{H1}( 70, 0){$E$}
\point{H2}( 105, 0){$X$      .}
\point{H3}( 87, -0.7){$\hooklongrightarrow$} 
\point{H4}( 87,  6){$i$}
\arrow{H0}{H1} \midput( -25 ,5){$can. lifting$}
\end{picture}
$$
\\
(Since $\varphi$ is an isomorphism, the canonical lifting of $u$ is determined.) Then we obtain the dotted arrow making the diagram commute since the left vertical arrow is an anodyne extension. Its restriction to ${\Delta}^n \times (0)$ is a solution of the original lifting problem.
%{\hfill $\Box$ \endtrivlist}
\end{proof}

Since $\dot{\Delta^n}$ is sequentially small, we can apply the small object 
argument ([3, p. 104]) to ${\mathscr F^{\prime}} = \{\dot{\Delta^n} \longrightarrow {\Delta}^n \hspace{1mm}(n \geq 0)\}$. Then we obtain a factorization
$$
\begin{picture}(70,70)(-5,-50)
\point{H0}( 0, 0){$X$}
\point{H1}( 70, 0){$G^{\infty}(\mathscr F^{\prime},f)$}
\point{H2}( 70,-70){$Y$}
\point{H3}(150,-20){$.$}
\arrow{H0}{H1} \midput(  0, 3){\hss{$j_{\infty}$}\hss}
\arrow{H0}{H2} \midput( -5, -5){\cbox{$f$}}
\arrow{H1}{H2} \midput( 0, -5){\hspace{1mm}{$q_{\infty}$}}
\end{picture}
$$
\\
\newline
such that $q_{\infty}$ is an acyclic fibration. Note that ${\dot{\Delta^n} \longrightarrow {\Delta}^n }$ is a cofibration by Lemma 2.4, and observe 
that the class of monomorphisms which are cofibrations are saturated 
(cf. [4, Lemma 4.1 in Chapter I]). Then it is seen that $j_{\infty}$ is a
cofibration. Thus we have the desired factorization, which proves the axiom ${\bf MC5}$(i).

${\bf MC4}$(i) is immediate from the definition of cofibrations. For ${\bf MC4}$(ii), suppose that $i:A \longrightarrow B$ is an acyclic cofibration; we have to show that $i$ has the left lifting property with respect to fibrations. Use the construction in the proof of ${\bf MC5}$(ii) to factor $i:A \longrightarrow B$ as a composite
$$
\begin{picture}(0,0)(40,0)
\point{H0}( 0, 0){$A$}
\point{H1}( 50, 0){$A^{\prime}$}
\point{H2}( 100, 0){$B$}
\arrow{H0}{H1} \midput( -5 ,5){$i_{\infty}$}
\arrow{H1}{H2} \midput( -5, 5){$p_{\infty}$}
\end{picture}
$$
and consider the following commutative solid arrrow diagram

$$
\xymatrix @r@R=13mm @r@C=17mm{
A \ar[d]^i \ar[r]^{i_{\infty}} \ar@{}[dr]   & A^{\prime} \ar[d]^{p_{\infty}} \\
B \ar@^{-->}[ur]^l \ar[r]^{\hspace{3mm}{1_B}} & B.\\
}
$$
Since $p_{\infty}$ is an acyclic fibration by ${\bf MC2}$, the dotted arrow $l$ exists, making the diagram commue. Thus we see that $i$ is a retract of $i_{\infty}$.
%$$
%\begin{picture}(70,70)(20,-50)
%\point{H0}( 0, 0){$A$}
%\point{H1}( 50, 0){$A$}
%\point{H2}( 100, 0){$A$}
%\point{H3}( 0, -50 ){$B$}
%\point{H4}( 50,-50 ){$A^{\prime}$}
%\point{H5}(100,-50 ){$B$}
%\point{H6}(120,-53){.}
%\arrow{H0}{H1} \midput( -5 ,5){$1_A$}
%\arrow{H1}{H2} \midput( -5, 5){$1_A$}
%\arrow{H0}{H3} \midput( 3, 0){$i$}
%\arrow{H1}{H4} \midput( 3, 0){$i_{\infty}$}
%\arrow{H2}{H5} \midput( 3, 0){$i$}
%\arrow{H3}{H4} \midput( -5, 5){$l$}
%\arrow{H4}{H5} \midput( -5, 5){$p_{\infty}$}
%\end{picture}
%$$
 As was seen in the proof of ${\bf MC5}$(ii), $i_{\infty}$ is an anodyne 
extension, and hence so is $i$. Thus $i$ has the left lifting property with respect to fibrations (cf. [4, Corollary 4.3 in Chapter I]).
\setcounter{thmcounter}{1}
\begin{rem}
%\begin{itshape}
%\begin{flushleft}
In the proof of ${\bf MC4}$(ii), we have shown that any acyclic cofibration is
an anodyne extension. %Note that an anodyne extension has the left lifting 
%property with respect to all fibrations ([G-J, Corollary 4.3 in Chapter I]),
%and that the acyclic cofibrations are the maps which have the left lifting
%property with respect to fibrations ([D-S, Proposition 3.13]).
%Then we see that any anodyne extension is an acyclic cofibration. 
The converse is shown in Lemma 2.1. Therefore the class of acyclic cofibrations in the model category ${\bf S}$ is just the class of anodyne extensions. 
%\end{flushleft}
%\end{itshape}
\end{rem}

\section{Cofibrations and weak equivalences in $\bf S$}
Let us begin by identifying the cofibrations in the model category $\bf S$.
We denote by $\emptyset$ an initial object in the category $\bf S$. A simplicial set
$X$ is said to be ${\it cofibrant}$ if the map from $\emptyset$ to $X$ is a cofibration.
\stepcounter{thmcounter}
\begin{prop}
%\begin{itshape}
%\begin{flushleft}
%$\bf Proposition$ $\bf 3.1.$ 
Let $i:A \longrightarrow B$ be a map in $\bf S$. Then the following are equivalent:
\newline
\newline
\hspace{2.5em}${\rm (i)}$ $i$ is a cofibration.

\hspace{1.3em}${\rm (ii)}$ $i$ is a monomorphism.

\hspace{1.1em}${\rm (iii)}$ $i$ is a degreewise injection.
\newline
\newline
In particular, any simplicial set is cofibrant.
\end{prop}
%\end{flushleft}
%\end{itshape}
\begin{proof}
The equivalence of (ii) and (iii) is obvious. For the equivalence of (i) and (iii), the same arguments as in [13, 2.1 and 3.15 in Chapter II] apply.
\end{proof}
%\begin{flushleft}
$\textit Proof$ $\textit of$ $\textit Proposition$ $\textit 1.2.$ In a model category, the class of fibrations and the class of cofibrations determine the class of weak equivalences ([5, Proposition 7.2.7]). Thus Proposition 3.1 and the definition of fibrations imply that our model category structure on $\bf S$ coincides with those of Quillen and Goerss-Jardine. %the oridinary one (c.f. [11], [4]).
{\hfill $\Box$ \endtrivlist}
\vspace{5mm}
%\end{flushleft}
%In a model category, the class of fibrations and the class of cofibrations determine the class of weak equivalences ([5, Proposition 7.2.7]). Thus Proposition 3.1 and the definition of fibrations imply that our model category structure on $\bf S$ coincides with the oridinary one (c.f. [11], [4]).

Next, let us identify the weak equivalences in $\bf S$. A map ${f: X \longrightarrow Y}$
between Kan complexes is called a weak homotopy equivalence if ${f_{\sharp} : \pi_0(X) \longrightarrow \pi_0(Y)}$ is a bijection and ${f_{\sharp} : \pi_i(X,x) \longrightarrow \pi_i(Y,f(x))}$ is an isomorphism for any 0-simplex $x$ of $X$ and any $i$ $>$ 0.
\newline
\begin{prop}
%\begin{itshape}
%\begin{flushleft}
%$\bf Proposition$ $\bf 3.2.$ 
Let $f:X \longrightarrow Y$ be a map in $\bf S$. Then the following are equivalent:
\newline
\newline
\hspace{2.5em}${\rm (i)}$ $f$ is a weak equivalence.

\hspace{1.3em}${\rm (ii)}$ $ \tilde{f}:\tilde{X} \longrightarrow \tilde{Y}$ is a homotopy equivalence.

\hspace{1.1em}${\rm (iii)}$ $ \tilde{f}:\tilde{X} \longrightarrow \tilde{Y}$ is a weak homotopy equivalence.
%\end{flushleft}
%\end{itshape}
\end{prop}
%\begin{proof}
By Lemma 2.3, it is enough to show the following lemma. It is a simplicial analogue of the Whitehead theorem and is proven in [8, $\S$12]. But we give a simple proof using the model structure of the category $\bf S$.%${\it Proof.}$ \hspace{0.3em}By Lemma 2.3, it is enough to show the following lemma.
%\end{proof}
\begin{lem}
%\begin{itshape}
%\begin{flushleft}
%$\bf Lemma$ $\bf 3.3.$ 
Let ${f : X \longrightarrow Y}$ be a map between Kan complexes. Then the following are equivalent:
\newline
\newline
\hspace{0.25em}$\rm (i)$ ${f : X \longrightarrow Y}$ is a homotopy 
equivalence.\\
$\rm (ii)$ ${f : X \longrightarrow Y}$ is a weak homotopy equivalence.
%\end{flushleft}
%\end{itshape}
\end{lem}
\begin{proof}
(i) $\Longrightarrow$ (ii) Obvious.\\
(ii) $\Longrightarrow$ (i) By the functorial factorization of ${\bf MC5}$(ii) and Lemma 2.3, we can assume that ${f}$ : ${X} \longrightarrow {Y}$ is a fibration. Take a minimal subfibration $\varphi$ : $E \longrightarrow {Y}$ of ${f}$ and a vertical deformation retraction ${r : X \longrightarrow E}$. Since the fiber $M$ of $\varphi$ is a minimal complex whose homotopy groups are trivial, $M$ is a terminal object, hence $\varphi$ : $E \longrightarrow{Y}$ is an isomorphism. Note that ${f = \varphi r}$ holds, and that $r$ is a homotopy equivalence. Then it follows that $f$ is a homotopy equivalence.
%{\hfill $\Box$ \endtrivlist}
\end{proof}

\section{Alternative proof of Lemma 2.4}
\begin{proof}
($\Longleftarrow$) Suppose that $f$ has the right lifting property with respect to \{$\dot{\Delta}^n \longrightarrow \Delta^n (n \geq 0)$\}. From the skeletal decomposition of a degreewise injection (the relative version of the decomposition in [4, p. 8])
%the proof of (iii) $\Longrightarrow$ (i) in Proposition 3.1
 and [4, Lemma 4.1 in Chapter I], $f$ has the right lifting property with respect to the class of degreewise injections. From this, let us see that $f$ is an acyclic fibration.

It is easily seen that $f$ is a fibration. In order to see that $f$ is a weak equivalence, take a solution $s$ to a lifting problem 
$$
\xymatrix @r@R=13mm @r@C=17mm{
\emptyset \ar[d] \ar[r] & X \ar[d]^f\\
Y \ar@^{-->}^s [ur] \ar[r]^{1_Y} & Y{,}\\
}
$$
which gives rise to a retract diagram
$$Y \overset{s} \longrightarrow X \overset{f} \longrightarrow Y$$
in the overcategory ${\bf S}/Y$. Next, take a solution $h$ to a lifting problem
$$
\xymatrix @r@R=13mm @r@C=19mm{
X \times \dot{I} \ar[d] \ar[r]^{1_X + sf} & X \ar[d]^f\\
X \times I \ar@^{-->}^h [ur] \ar[r]^{\hspace{3mm}f \hspace{0.5mm} \circ \hspace{0.5mm} proj} & Y ,\\
}
$$
which shows that $Y$ is a deformation retract of $X$ over $Y$. Thus $f$ is a homotopy equivalence and hence a weak equivalence.\\
($\Longrightarrow$) Suppose that $f$ is an acyclic fibration and ${\varphi}^{\prime}$ : $E^{\prime} \longrightarrow \tilde{Y}$ is the $M$-bundle constructed in the original proof. First let us show that there exists a solution $l$ to a lifting problem of the form
$$
\xymatrix @r@R=13mm @r@C=17mm{
\dot{\Delta}^n \ar[d] \ar[r]^v & E^{\prime} \ar[d]^{{\varphi}^{\prime}}\\
{\Delta}^n \ar@^{-->} [ur]^l \ar[r]^{u} & \hspace{1mm}\tilde{Y}.  \\
}
$$
Since ${\varphi}^{\prime}$ is surjective, we can take a lifting $l_0$ of $u$. Since ${\varphi}^{\prime}$ is a homotopy equivalence by Lemma 2.3, we have $l_0 \mid_{\dot{\Delta}^n} \simeq v$. Thus we can take a 1-simplex $m_0$ of the function complex ${\bf Hom}$($\dot{\Delta}^n$,$E^{\prime}$) connecting $l_0 \mid_{\dot{\Delta}^n}$ to $v$. Note that $\varphi^{\prime}_\sharp$ : ${\bf Hom}$($\dot{\Delta}^n$,$E^{\prime}$) $\longrightarrow$ ${\bf Hom}$($\dot{\Delta}^n$,$\tilde{Y}$) is a fibration which is a homotopy equivalence between Kan complexes (cf. [4, p. 21])%. By considering $\pi_1(\varphi^{\prime}_{\sharp}$), we can retake a 1-simplex of ${\bf Hom}$($\dot{\Delta}^n$, $E^{\prime}$) so that it gives a vertical homotopy $m$ from $l_0 \mid_{\dot{\Delta}^n}$ to $v$. Then a lifting problem
, and that ${a :=}$ $\varphi^{\prime}_\sharp(m_0)$ determines an element of 
${\pi_1({\bf Hom}(\dot{\Delta}^n,\tilde{Y}), \hspace{1mm}u\mid_{\dot{\Delta}^n})}$. Thus we take a representative $\tilde{b}$ of 
${\pi_1(\varphi^\prime_\sharp)^{-1}([a]^{-1}) \in \pi_1({\bf Hom}(\dot{\Delta}^n,E^{\prime}),v)}$ and set 
${b = \varphi^\prime_\sharp(\tilde{b})}$. Then the map
$$(b, 0, a) : \dot{\Delta}^2 \longrightarrow {\bf Hom}(\dot{\Delta}^n,\tilde{Y})$$
has a filler ${\gamma : \Delta^2 \longrightarrow {\bf Hom}(\dot{\Delta}^2,\tilde{Y})}$, where ${(b, 0, a)}$ denotes the map whose restrictions to the 0-th, 
1st, and 2nd faces are $b$, the constant map to $u\mid_{\dot{\Delta}^n}$, and
$a$ respectively. Let $\beta$ denote the sum of $m_0$ and $\tilde{b}$, and
consider the lifting problem
$$
\xymatrix @r@R=13mm @r@C=20mm{
\Lambda^2_1 \ar[d] \ar[r]^{\hspace{-1cm}\beta} & {\bf Hom}(\dot{\Delta}^n, E^\prime) \ar[d]^{{\varphi}^\prime_\sharp}\\
{\Delta}^2\ar@^{-->} [ur]^\delta \ar[r]_{\hspace{-1cm}\gamma} & {\bf Hom}(\dot{\Delta}^n, \tilde{Y}).\\
}
$$
Since $\varphi^\prime_\sharp$ is a fibration, there exists a solution $\delta$.
Then the restriction of $\delta$ to the first face gives a vertical homotopy
$m$ connecting $\ell_0 \mid_{\dot{\Delta}^n}$ to $v$. 

Next, consider the lifting problem

$$
\xymatrix @r@R=13mm @r@C=20mm{
{\Delta}^n \times (0) \cup \dot{\Delta}^n \times I \ar[d] \ar[r]^{\hspace{1cm}l_0 + m} & E^{\prime} \ar[d]^{{\varphi}^{\prime}}\\
{\Delta}^n \times I \ar@^{-->} [ur]^h \ar[r]_{u\hspace{0.5mm}\circ \hspace{0.5mm} proj} & \tilde{Y}.\\
}
$$
It has a solution $h$ since the left vertical arrow is an anodyne extension. Set $l$ = $h\mid_{{\Delta}^n \times (1)}$, which is a solution of the original problem.\\
\hspace{1em}We can show that $f$ has the right lifting property with respect to $\{ \dot{\Delta}^n \longrightarrow 
\Delta^n (n \geq 0)\}$ by a similar argument to that used in the original proof.
%{\hfill $\Box$ \endtrivlist}
\end{proof}
$\bf Acknowledgments$ The author would like to thank Prof. N. Oda for his valuable comments on an earlier version of this paper.


\begin{thebibliography}{99}
\bibitem{no1} Barratt, M.G., Guggenheim, V.K.A.M., Moore, J.C.: On semisimplicial fibre bundles. Am. J. Math. 81, 639-657 (1959)
\bibitem{no3} Ciampella, A.: The complete Steenrod algebra at odd primes, Ricerche di Matematica 57, no.1, 65-79 (2008)
%\bibitem{no2} Bousfield,A.K., Kan,D.M.: Homotopy Limits, Completions and Localizations, Lecture Notes in Math.  Springer-Verlag,304 (2nd corrected printing)(1987)
\bibitem{no2}Dwyer, W.G and Spalinski, J.: Homotopy theories and model categories. In: James, I.M. (eds.) Handbook of algebraic topology, pp. 73-126, North-Holland, North Holland (1995)
%\bibitem{no2}  W.~G.~Dwyer and J.~Spalinski,
%``Homotopy theories and model categories,''
%In Handbook of algebraic topology, North-Holland, 1995, 73--126.

%\bibitem{no3} S.~I.~ Gelfand and Yu.~I.~Manin,
%''Methods of Homological Algebra,''
%Springer-Verlag, Berlin, 1996.


\bibitem{no4} Goerss, P.G., Jardine, J.F.: Simplicial Homotopy Theory. Progress in Mathematics, Birkhauser, Basel-Boston-Berlin (1999)
%Goerss, P.G., Jardine, J.F.: Simplicial Homotopy Theory, Progress in Mathematics, Birkhauser, Basel-Boston-Berlin, 174 (1999)

\bibitem{no5} Hirschhorn, P.S.: Model categories and their localizations. Mathematical Surveys and Monographs, American Mathematical Society, Providence (2003)
%Hirschhorn, P.S.: Model categories and their localizations, Mathematical Surveys Monographs, American Mathematical Society Vol.99, Providence (2003)

\bibitem{no6} Hovey, M.: Model categories. Mathematical Surveys and Monographs, American Mathematical Society, Providence (1998)
%Hovey, M., Model categories, Mathematical Surveys and Monographs,  American Mathematical Society, Vol. 63,:  Providence (1998)

\bibitem{no7} L${\rm{\acute{a}}}$russon, F.: Model structures and the Oka principle. J. Pure Appl. Algebra 192, 203-223 (2004)


\bibitem{no7} May, J.P.: Simplicial Objects in Algebraic Topology. Van Nostrand, Princeton (1968)

\bibitem{no10} May, J.P., Ponto, K.: More Concise Algebraic Topology: localization, completion, and model categories, University of Chicago Press (2012)

\bibitem{no8} Morel, F., Voevodsky, V.: $\mathbb{A}^{1}$-homotopy theory of schemes. Inst. Hautes Etudes Sci. Publ. Math. 90, 45-143 (2001)


\bibitem{no8} To${\rm{\ddot{e}}}$n, B., Vezzosi, G.: Homotopical algebraic geometry I: Topos theory. Adv. in Math. 193, 257-372 (2005)
%To${\rm{\ddot{e}}}$n, B., Vezzosi, G.: Homotopical algebraic geometry I: Topos theory, Adv. in Math. 193 , no.2, 257-372 (2005)


\bibitem{no9} To${\rm{\ddot{e}}}$n, B., Vezzosi, G.: Homotopical algebraic geometry II: Geometric stacks and applications. American Mathematical Society, Providence (2008)
%To${\rm{\ddot{e}}}$n, B., Vezzosi, G.: Homotopical algebraic geometry II: Geometric stacks and applications, ArXiv:math.AG/0404373.

\bibitem{no9} Quillen, D.G.: Homotopical Algebra. Lecture Notes in Math. 76, Springer-Verlag (1968)


%\bibitem{no9}  D.~G.~Quillen,
%''Rational homotopy theory,''
%Ann. Math.,
%vol.90, 1969, 205--295.
\end{thebibliography}
\end{document}